\documentclass{article}
\usepackage{amsmath,amssymb,amsthm}
\usepackage{accents}
\usepackage{bm}
\usepackage{xcolor}
\usepackage[top=25truemm,bottom=25truemm,left=25truemm,right=25truemm]{geometry}
\usepackage{graphicx}
\usepackage{hyperref}
\usepackage{mathtools}
\usepackage{url}
\theoremstyle{definition}
\newtheorem{Def}{Definition}[section]
\newtheorem{Lem}[Def]{Lemma}
\newtheorem{Prop}[Def]{Proposition}
\newtheorem{Thm}[Def]{Theorem}

\newtheorem{Rmk}[Def]{Remark}
\newtheorem{Alg}[Def]{Algorithm}
\newtheorem{Prob}[Def]{Problem}

\makeatletter
\def\bar{\accentset{{\cc@style\underline{\mskip13mu}}}}

\makeatother
\title{{\bf\textsf{On superspecial hyperelliptic curves of Rosenhain forms}}}
\author{Ryo Ohashi}

\begin{document}
\maketitle
\begin{abstract}
Any genus-$g$ hyperelliptic curve $C$ defined over an algebraically closed field of characteristic $p \geq 3$ can be written in a Rosenhain form as $y^2 = x(x-1)\prod_{i=1}^{2g-1}(x-\lambda_i)$.
In this paper, we first show that, if $C$ is superspecial, then each of $\lambda_i,1-\lambda_i$, and $\lambda_i-\lambda_j$ is a square in $\mathbb{F}_{p^2}$.
As an application, we propose a new algorithm for enumerating superspecial hyperelliptic curves in small characteristic.
By implementing our algorithm, we successfully computed the number of isomorphism classes of such curves of genera $4$ and $5$ in all characteristics $p \leq 41$, and of genus $6$ in all characteristics $p \leq 31$.
\end{abstract}

\section{Introduction}
Throughout this paper, a \emph{curve} always means a non-singular projective variety of dimension one defined over a field of characteristic $p > 0$.
All \emph{isomorphisms} are taken over the algebraic closure of the base field, unless otherwise stated.
A curve is said to be \emph{superspecial} if its Jacobian is isomorphic to a product of supersingular elliptic curves (as unpolarized abelian varieties).
In recent years, superspecial curves have found applications in areas such as isogeny-based cryptography and algebraic geometry codes.
Consequently, understanding the properties of such curves has become an important research topic.
For a positive integer $g$ and a prime $p$, it is known (cf. \cite[Corollary 1.2]{NN}) that the number of isomorphism classes of superspecial curves of genus $g$ in characteristic $p$ is finite, and the following is a fundamental problem:\vspace{-0.8mm}
\begin{Prob}\label{prob:ssp}
For a given pair $(g,p)$, determine the number of isomorphism classes of superspecial genus-$g$ curves in characteristic $p$.\vspace{-0.2mm}
\end{Prob}
\noindent This problem has been solved for genera $g \leq 3$, but for $g \geq 4$, even the existence of such curves remains open in general $p$.
We refer to \cite{Kudo} for a survey of related works.

On the other hand, the focus of this paper is restricted to \emph{hyperelliptic} curves.
More precisely, we consider the following problem:\vspace{-0.8mm}
\begin{Prob}\label{prob:hyp}
For a given pair $(g,p)$, determine the number of isomorphism classes of superspecial genus-$g$ hyperelliptic curves in characteristic $p$.\vspace{-0.2mm}
\end{Prob}
\noindent Since every genus-2 curve is hyperelliptic, Problem~\ref{prob:hyp} is exactly the same as Problem~\ref{prob:ssp} for the case $g=2$, and hence has already been solved, as mentioned above.
However, for every $\hspace{-0.2mm}g \geq 3$, Problem~\ref{prob:hyp} remains open in general $p$.
Nevertheless, for small characteristics $p$, the answers to Problem~\ref{prob:hyp} are known in several cases, as summarized below:\vspace{-0.5mm}
\begin{itemize}
\item Ekedahl showed in \cite[Theorem 1.1]{Ekedahl} that if there exists a superspecial hyperelliptic curve of genus $g \geq 2$ in characteristic $p$, then $2g \leq p-1$ holds.
Furthermore, \hspace{-0.3mm}Valentini proved in \cite[Theorem 1]{Valentini} that there exists a unique isomorphism class of such curves, namely the curve $y^2 = x^p-x$, when $2g=p-1$ holds.
In other words, the answer to Problem~\ref{prob:hyp} is $0$ for $p \leq 2g-1$, and is $1$ for $p=2g+1$.
\vspace{-1.1mm}
\item In the case of genus $g=3$, the number of isomorphism classes for every characteristic $7 < p < 100$ was determined in \cite{OOKYN} by using the superspecial Richelot isogeny graph.\vspace{-1mm}
\item In the case of genus $g=4$, Kudo and Harashita determined in \cite{KH18,KH22} the number of isomorphism classes for the characteristics $p = 11,13,17$, and $19$ by using Gröbner basis computations.
\end{itemize}
To the best of our knowledge, for genus $g \in\hspace{-0.3mm} \{5,6\}$, the number of isomorphism classes has not been determined for any characteristic $p > 2g+1$.

\newpage
The main purpose of this paper is to extend the known results on Problem~\ref{prob:hyp} by determining the number of isomorphism classes of superspecial hyperelliptic curves of genera $4,5$, and $6$ in characteristics beyond the range previously covered in the literature.
To achieve this goal, we first establish the following useful theorem, which provides a necessary condition for hyperelliptic curves in \emph{Rosenhain forms} to be superspecial:\vspace{-1mm}
\begin{Thm}\label{thm:criterion}
Assume that $p > 2$. If a genus-$g$ hyperelliptic curve\vspace{-1.3mm}
\begin{equation*}\label{eq:C}
    C: y^2 = x(x-1)\prod_{i=1}^{2g-1}(x-\lambda_i)\vspace{-0.3mm}
\end{equation*}
in characteristic $p$ is superspecial, then all of $\lambda_i,1-\lambda_i,\lambda_i-\lambda_j$ are squares in $\mathbb{F}_{p^2}\hspace{-0.2mm}$ for all $i,j \in \{1,\ldots,2g-1\}$.\vspace{-0.5mm}
\end{Thm}
\noindent This theorem can be seen as a generalization of \cite[Main \hspace{-0.2mm}Theorem \hspace{-0.3mm}A(1)]{Ohashi}, which corresponds to the case $g=2$.
As an application of this theorem, we propose a new algorithm (Algorithm \ref{alg:main}) for enumerating superspecial genus-$g$ hyperelliptic curves in characteristic $p$ for fixed $(g,p)$.
Essentially, the algorithm consists of collecting all sets $\{\lambda_1,\ldots,\lambda_{2g-1}\}$ satisfying the condition of Theorem~\ref{thm:criterion} and then checking whether the Cartier-Manin matrix of the corresponding curve is zero (for details on the Cartier-Manin matrix, see \cite[Section 2]{Yui}).
Our approach has the advantage of avoiding Gröbner basis computations and allowing isomorphism testing to be performed relatively easily.
For small characteristics $p$, since the number of candidate sets $\{\lambda_1,\ldots,\lambda_{2g-1}\}$ is sufficiently small, our algorithm is computationally feasible.
By implementing and running the algorithm in \textsf{Magma}, we obtain the following main results:\vspace{-1mm}
\begin{Thm}\label{thm:genus4}
The number of isomorphism classes of superspecial genus-$4$ hyperelliptic curves is as follows:\vspace{-0.7mm}
\begin{enumerate}
\item[(1)] There are precisely $4$ isomorphism classes of such curves in characteristic $23$.\vspace{-2.2mm}
\item[(2)] There are precisely $8$ isomorphism classes of such curves in characteristic $29$.\vspace{-2.2mm}
\item[(3)] There are precisely $10$ isomorphism classes of such curves in characteristic $31$.\vspace{-2.2mm}
\item[(4)] There are precisely $23$ isomorphism classes of such curves in characteristic $37$.\vspace{-2.2mm}
\item[(5)] There are precisely $34$ isomorphism classes of such curves in characteristic $41$.
\end{enumerate}
\end{Thm}
\begin{Thm}\label{thm:genus5}
The number of isomorphism classes of superspecial genus-$5$ hyperelliptic curves is as follows:\vspace{-0.7mm}
\begin{enumerate}
\item[(1)] There are no such curves in characteristics $13$ and $17$.\vspace{-2.2mm}
\item[(2)] There is a unique isomorphism class of such curves in characteristics $19$ and $29$.\vspace{-2.2mm}
\item[(3)] There are precisely $2$ isomorphism classes of such curves in characteristic $23$.\vspace{-2.2mm}
\item[(4)] There are precisely $6$ isomorphism classes of such curves in characteristic $31$.\vspace{-2.2mm}
\item[(5)] There are precisely $5$ isomorphism classes of such curves in characteristic $37$.\vspace{-2.2mm}
\item[(6)] There are precisely $3$ isomorphism classes of such curves in characteristic $41$.
\end{enumerate}
\end{Thm}
\begin{Thm}\label{thm:genus6}
The number of isomorphism classes of superspecial genus-$6$ hyperelliptic curves is as follows:\vspace{-0.9mm}
\begin{enumerate}
\item[(1)] There are no such curves in characteristics $17,19$, and $29$.\vspace{-2.5mm}
\item[(2)] There is a unique isomorphism class of such curves in characteristics $23$ and $31$.
\end{enumerate}
\end{Thm}
\noindent For details of the computational experiments, see Section \ref{sec:algorithm}.
The above results can be summarized in Table~\ref{tbl:summary}, where our contributions are highlighted in \textcolor{magenta}{magenta}:\vspace{-2.6mm}
\begin{table}[h]
    \centering
    \caption{The number of isomorphic classes of superspecial genus-$g$ hyperelliptic curves in characteristic $p$}\label{tbl:summary}\vspace{2mm}
\begin{tabular}{c|c|c|c|c|c|c|c|c|c|c|c|c}
        $\!\!g \hspace{0.1mm}\backslash\hspace{0.5mm} p\!$ & $\leq 5$ & \hspace{0.8mm}$7$\hspace{0.8mm} & $11$ & $13$ & $17$ & $19$ & $23$ & $29$ & $31$ & $37$ & $41$ & $\cdots$\\\hline
        $3$ & 0 & 1 & 1 & 1 & 2 & 4 & 9 & 10 & 27 & 35 & 54 & $\cdots$\\\hline
        $4$ & 0 & 0 & 0 & 0 & 2 & 2 & \textcolor{magenta}{4} & \textcolor{magenta}{8} & \textcolor{magenta}{10} & \textcolor{magenta}{23} & \textcolor{magenta}{34} & \\\hline
        $5$ & 0 & 0 & 1 & \textcolor{magenta}{0} & \textcolor{magenta}{0} & \textcolor{magenta}{1} & \textcolor{magenta}{2} & \textcolor{magenta}{1} & \textcolor{magenta}{6} & \textcolor{magenta}{5} & \textcolor{magenta}{3} & \\\hline
        $6$ & 0 & 0 & 0 & 1 & \textcolor{magenta}{0} & \textcolor{magenta}{0} & \textcolor{magenta}{1} & \textcolor{magenta}{0} & \textcolor{magenta}{1} & & & 
    \end{tabular}
\end{table}\vspace{-4.8mm}

\paragraph{Organization.}
The remainder of this paper is organized as follows.
In Section \ref{sec:preliminaries}, after reviewing preliminaries on hyperelliptic curves, we prove Theorem \ref{thm:criterion}.
In Section \ref{sec:algorithm}, we present the main algorithm and report the computational results obtained from its implementation, namely \hspace{-0.2mm}Theorems \ref{thm:genus4}, \ref{thm:genus5}, and \ref{thm:genus6}.\vspace{-2.8mm}

\paragraph{Acknowledgements.} This research was supported by JSPS Grant-in-Aid for Young Scientists 25K17225.

\newpage
\section{Hyperelliptic curves of Rosenhain forms}\label{sec:preliminaries}
Let $K$ be a field of characteristic $p > 2$.
In this section, we shall consider a hyperelliptic curve of genus $g \geq 2$ defined by the affine equation\vspace{-2.8mm}
\begin{equation}\label{eq:H}
    H: Y^2 = c\prod_{i=1}^{2g+1}(X-a_i) \eqqcolon f(X), \quad c \in K^\times,\vspace{-1mm}
\end{equation}
where the $a_i$ are distinct elements of $\bar{K}$ \hspace{-0.3mm}(we also set $a_{2g+2} \coloneqq \infty$).
For any permutation $\sigma$ of $\{1,\ldots,2g+2\}$, the Möbius transformation\vspace{-1.7mm}
\[
    X \longmapsto \frac{X-a_{\sigma(2g)}}{X-a_{\sigma(2g+2)}} \cdot \frac{a_{\sigma(2g+1)}-a_{\sigma(2g+2)}}{a_{\sigma(2g+1)}-a_{\sigma(2g)}} \eqqcolon x
\]
maps $a_{\sigma(2g)},a_{\sigma(2g+1)}$, and $a_{\sigma(2g+2)}$ to $0,1$, and $\infty$, respectively.
After a suitable change of the $Y$-coordinate, it then induces an isomorphism over $\bar{K}$ from $H$ to its \emph{Rosenhain form}\vspace{-1.1mm}
\[
    y^2 = x(x-1)\prod_{i=1}^{2g-1}(x-\lambda_i), \quad \lambda_i \coloneqq \frac{a_{\sigma(i)}-a_{\sigma(2g)}}{a_{\sigma(i)}-a_{\sigma(2g+2)}} \cdot \frac{a_{\sigma(2g+1)}-a_{\sigma(2g+2)}}{a_{\sigma(2g+1)}-a_{\sigma(2g)}}.\vspace{-0.4mm}
\]
Conversely, every Rosenhain form of $H$ arises from a permutation $\sigma$ of $\{1,\ldots,2g+2\}$ in this way.\vspace{-1.1mm}
\begin{Rmk}
Since the Rosenhain form obtained above depends only on the images of $2g,2g\hspace{-0.1mm}+\hspace{-0.1mm}1$, and $2g\hspace{-0.1mm}+\hspace{-0.1mm}2$ under $\sigma$, the number of distinct Rosenhain forms of $H$ is at most $(2g\hspace{-0.1mm}+\hspace{-0.1mm}2) \cdot (2g\hspace{-0.1mm}+\hspace{-0.1mm}1) \cdot (2g) = 4g(g+1)(2g+1)$.
\end{Rmk}

Next, for each $i \in \{1,\ldots,2g+1\}$, let $P_i \coloneqq (a_i,0)$ be a \hspace{-0.3mm}Weierstrass point of $H$.
Then, the divisor class\vspace{-0.9mm}
\begin{equation}\label{eq:D_i}
    D_i \coloneqq [P_i] - [\infty] \ \text{ with }\,i \in \{1,\ldots,2g+1\}
\end{equation}
is a $2$-torsion point of the Jacobian $J$ of $H$, and its Mumford representation is given by $[X-a_i,0]$.
Moreover, the Mumford representations of the half-points of $D_i$ are described as follows (in particular, when $c=1$, the resulting formula coincides with \cite[Theorem 3.2]{Zarhin}).\vspace{-1.3mm}
\begin{Lem}\label{lem:half-point}
Fix $i \in\hspace{-0.1mm} \{1,\ldots,2g+1\}$.
With the above notation, for each $\ell \in \{1,\ldots,2g+1\}\!\smallsetminus\!\{i\}$, let $r_\ell \in \bar{K}^\times\hspace{-0.3mm}$ be an arbitrarily chosen square root of $c(a_i-a_\ell)$.
Define the tuple $\bm{r} \coloneqq (r_\ell)_{\ell \neq i}$ and\vspace{-0.6mm}
\begin{align}
    u_i(X) &\coloneqq \frac{(-1)^g}{c^g}\sum_{k=0}^{g}e_{2k}(\bm{r})\bigl(-c(X-a_i)\bigr)^{g-k},\label{eq:u_i}\\
    v_i(X) &\coloneqq \frac{1}{c^g}\sum_{k=1}^ge_{2k-1}(\bm{r})\bigl(-c(X-a_i)\bigr)^{g-k+1},\nonumber
\end{align}
where $e_k$ denotes the $k$-th elementary symmetric polynomial.
Then $[u_i(X),v_i(X) \bmod u_i(X)]$ is a Mumford representation of a point $T_i \in J(\bar{K})$ satisfying $2T_i = D_i$.\vspace{-1mm}
\end{Lem}
\begin{proof}
Set $t \coloneqq -c(X-a_i)$, and put\vspace{-0.8mm}
\[
    U(t) \coloneqq \sum_{k=0}^{g}e_{2k}(\bm{r})t^{g-k}, \quad V(t) \coloneqq \sum_{k=1}^{g}e_{2k-1}(\bm{r})t^{g-k}.
\]
Let $z$ be a formal variable with $z^2=t$.
By definition of the elementary symmetric polynomials, we have\vspace{-0.4mm}
\begin{align*}
    \prod_{\ell \neq i}(z+r_\ell) &= \sum_{k=0}^{2g}e_k(\bm{r})z^{2g-k}\\[-0.5mm]
    &= \sum_{k=0}^ge_{2k}(\bm{r})z^{2g-2k} + \sum_{k=1}^ge_{2k-1}(\bm{r})z^{2g-2k+1}\\[-0.5mm]
    &= \sum_{k=0}^ge_{2k}(\bm{r})t^{g-k} + z\sum_{k=1}^ge_{2k-1}(\bm{r})t^{g-k} = U(t)+zV(t).\\[-7.5mm]
\end{align*}
This equation implies that\vspace{-1mm}
\begin{align*}
    U(t)^2-tV(t)^2 &= (U(t)+zV(t))(U(t)-zV(t))\\
    &= \prod_{\ell \neq i}(z+r_\ell) \cdot \prod_{\ell \neq i}(-z+r_\ell) = \prod_{\ell \neq i}(r_\ell^2-t).\\[-5.5mm]
\end{align*}
On the other hand, since $r_\ell^2 = c(a_i-a_\ell)$, we obtain $r_\ell^2-t = c(a_i-a_\ell) + c(X-a_i) = c(X-a_\ell)$, and hence, the above equation can be transformed into\vspace{-0.2mm}
\[
     U(t)^2-tV(t)^2 = c^{2g}\prod_{\ell \neq i}(X-a_\ell).
\]
Then, the right-hand side $f(X)$ of \eqref{eq:H} is written as\vspace{-0.8mm}
\[
    f(X) = c(X-a_i)\prod_{\ell \neq i}(X-a_\ell) = -\frac{t}{c^{2g}}(U(t)^2-tV(t)^2).\vspace{0.7mm}
\]
It follows from the definition of $U(t)$ and $V(t)$ that\vspace{-0.6mm}
\[
    u_i(X) = \frac{(-1)^g}{c^g}U(t), \quad v_i(X) = \frac{t}{c^g}V(t),
\]
which gives the key equation\vspace{-2.8mm}
\begin{align*}
    f(X) - v_i(X)^2 &= -\frac{t}{c^{2g}}(U(t)^2-tV(t)^2) - \frac{t^2}{c^{2g}}V(t)^2\\
    &= -\frac{t}{c^{2g}}U(t)^2\\
    &= c(X-a_i)u_i(X)^2.
\end{align*}
In particular, we see that $f(X)-v_i(X)^2$ is a multiple of $u_i(X)$.
Also, it is clear that\vspace{-0.5mm}
\begin{itemize}
    \item $u_i(X)$ is a monic polynomial of degree $g$,\vspace{-2.3mm}
    \item $v_i(X) \bmod{u_i(X)}$ has degree less than $g$,\vspace{0.8mm}
\end{itemize}
and thus $[u_i(X),v_i(X) \bmod{u_i(X)}]$ is a valid Mumford representation (cf. \cite[Theorem 13.7]{Washington}).
Furthermore, by Cantor's doubling algorithm (cf. \cite[Theorem 13.10]{Washington}), the $u$-polynomial of the double of the divisor class represented by $[u_i(X),v_i(X) \bmod{u_i(X)}]$ is given by $X-a_i$ from the key equation\vspace{-0.6mm}
\[
    \frac{f(X)-v_i(X)^2}{u_i(X)^2} = c(X-a_i).\vspace{-0.4mm}
\]
The corresponding $v$-polynomial is necessarily $0$ since $\deg(X-a_i) = 1$ and $f(a_i) = 0$ (cf. \cite[Example 13.2]{Washington}).
Consequently, we finally obtain\vspace{-0.7mm}
\[
    2[u_i(X),v_i(X) \bmod{u_i(X)}] = [X-a_i,0],\vspace{0.3mm}
\]
that is, $[u_i(X),v_i(X) \bmod{u_i(X)}]$ represents a point $T_i \in J(\bar{K})$ satisfying $2T_i = D_i$.\vspace{-1.2mm}
\end{proof}
\begin{Rmk}
In Lemma~\ref{lem:half-point}, each square root $r_\ell$ may be chosen independently up to sign.
Hence, there are $2^{2g}$ possible choices for the tuple $\bm{r} = (r_\ell)_{\ell \neq i}$, yielding $2^{2g}$ half-points of $D_i$.
This is consistent with the fact that $D_i$ has exactly $2^{2g}$ half points, since the $2$-torsion subgroup $J[2]$ has order $2^{2g}$.\vspace{0.5mm}
\end{Rmk}

Lemma~\ref{lem:half-point} provides explicit Mumford representations of certain 4-torsion points on the Jacobian $J$ of $H$.
We now prove the following proposition, which plays a central role in the proof of Theorem~\ref{thm:criterion}.\vspace{-1.1mm}
\begin{Prop}\label{prop:4-torsion}
Let $J$ be the Jacobian of the hyperelliptic curve $H$ as in \eqref{eq:H}.
Suppose that the $4$-torsion subgroup $J[4]$ is defined over $K$.
Then, for any pairwise distinct $i,j,k \in \{1,\ldots,2g+1\}$, the ratio\vspace{-0.1mm}
\begin{equation}\label{eq:ratio}
    \frac{a_i-a_j}{a_i-a_k}
\end{equation}
is a square in $K$.\vspace{-3mm}
\end{Prop}

\newpage
\begin{proof}
Since $J[4]$ is defined over $K$, the $2$-torsion subgroup $J[2]$ is also defined over $K$.
Hence, in particular, the divisor class $D_i$ defined by \eqref{eq:D_i} is $K$-rational for every $i \in \{1,\ldots,2g+1\}$.
Recalling that the Mumford representation of $D_i$ is $[X-a_i,0]$, we obtain
$a_i \in K$ for every $i \in \{1,\ldots,2g+1\}$.

Fix pairwise distinct $i,j,k \in \{1,\ldots,2g+1\}$.
For each $\ell \in \{1,\ldots,2g+1\}\!\smallsetminus\!\{i\}$, choose an arbitrary square root $r_\ell \in \bar{K}^\times\hspace{-0.3mm}$ of $c(a_i-a_\ell)$, and define $\bm{r} \coloneqq (r_\ell)_{\ell \neq i}$.
Also, let $\bm{r}^{(j)},\bm{r}^{(k)}$, and $\bm{r}^{(jk)}$ denote the tuples obtained from $\bm{r}$ by changing the sign of $r_j$, the sign of $r_k$, and the signs of both $r_j$ and $r_k$, respectively.
By Lemma~\ref{lem:half-point},\\each of these choices determines a half-point of $D_i$.
Since $J[4]$ is defined over $K$, all of these half-points are $K$-rational.
Hence, the corresponding $u$-polynomials have coefficients in $K$.
In particular, by considering the coefficients of $(X-a_i)^{g-1}$ in \eqref{eq:u_i}, we obtain\vspace{-1.6mm}
\begin{equation}\label{eq:rational}
    e_2(\bm{r}),\ e_2(\bm{r}^{(j)}),\ e_2(\bm{r}^{(k)}),\ e_2(\bm{r}^{(jk)}) \in K\vspace{-0.2mm}
\end{equation}
since $c \in K^\times$.
On the other hand, a direct computation gives\vspace{-1.2mm}
\begin{equation}\label{eq:e_2}
    e_2(\bm{r}) - e_2(\bm{r}^{(j)}) - e_2(\bm{r}^{(k)}) + e_2(\bm{r}^{(jk)}) = 4r_jr_k.\vspace{0.3mm}
\end{equation}
Indeed, for a monomial $r_ar_b$, its contribution to $e_2(\bm{r}) - e_2(\bm{r}^{(j)}) - e_2(\bm{r}^{(k)}) + e_2(\bm{r}^{(jk)})$ is as follows:\vspace{-1mm}
\begin{itemize}
    \item If $\{a,b\}=\{j,k\}$, then its contribution is $r_jr_k - (-r_j)r_k - r_j(-r_k) + (-r_j)(-r_k) = 4r_jr_k$.\vspace{-2.2mm}
    \item If exactly one of $a,b$ belongs to $\{j,k\}$, then its contribution is\vspace{-1.4mm}
    \[
        r_jr_b - (-r_j)r_b - r_jr_b + (-r_j)r_b = 0 \ \text{ or }\ r_ar_k - r_ar_k - r_a(-r_k) + r_a(-r_k) = 0.\vspace{-1.3mm}
    \]
    \item Otherwise, its contribution is $r_ar_b - r_ar_b - r_ar_b + r_ar_b = 0$.\vspace{-0.6mm}
\end{itemize}
Hence, all terms cancel except $4r_jr_k$, which proves \eqref{eq:e_2}.
By combining \eqref{eq:rational} and \eqref{eq:e_2}, we obtain $r_jr_k \in K$. Since $r_k^2 = c(a_i-a_k) \in K^\times$, it follows that\vspace{-2.6mm}
\[
    \frac{a_i-a_j}{a_i-a_k} = \frac{r_j^2}{r_k^2} = \biggl(\frac{r_jr_k}{r_k^2}\biggr)^{\!2}\vspace{-1.2mm}
\]
is a square in $K$, as desired.
\end{proof}\vspace{0.5mm}

At the end of this section, we give a proof of Theorem~\ref{thm:criterion}.
\begin{proof}[Proof of Theorem~\ref{thm:criterion}]
It is well-known (cf. \cite[p.\,166]{Ekedahl}) that $C$ descends to a curve $H$ over $\mathbb{F}_{p^2}\hspace{-0.2mm}$ such that the Frobenius endomorphism $F$ of its Jacobian $J$ acts as $\pm p$.
Then, we choose $H$ as follows:\vspace{-1mm}
\begin{itemize}
\item[---] If $p \equiv 3 \pmod{4}$, we choose $H$ such that $F$ is equal to $-p$.\vspace{-2.6mm}
\item[---] If $p \equiv 1 \pmod{4}$, we choose $H$ such that $F$ is equal to $p$.
\end{itemize}
In either case, the Frobenius endomorphism $F$ acts trivially on $J[4]$.
This shows that every point of $J[4]$ is fixed by $F$, and thus $J[4]$ is defined over $\mathbb{F}_{p^2}$.
As $H$ is defined over $\mathbb{F}_{p^2}$, it admits a model of the form\vspace{-1.4mm}
\[
    H:\ Y^2=c\prod_{i=1}^{2g+2}(X-a_i),
    \qquad c \in\mathbb F_{p^2}\vspace{-0.3mm}
\]
where all the $a_i$ are distinct elements of $\overline{\mathbb{F}}_p$.
Let $P = (a_{2g},0),\hspace{0.7mm}Q = (a_{2g+1},0)$, and $R = (a_{2g+2},0)$, which are Weierstrass points of $H$.
Since $J[2]$ is defined over $\mathbb{F}_{p^2}$, the divisor classes $[P]-[Q],\hspace{0.8mm}[Q]-[R]$, and $[R]-[P]$ are $\mathbb{F}_{p^2}$-rational. 
Hence, their Mumford representations are defined over $\mathbb{F}_{p^2}$, so the corresponding $u$-polynomials\vspace{-1.1mm}
\[
    (X-a_{2g})(X-a_{2g+1}), \quad 
    (X-a_{2g+1})(X-a_{2g+2}), \quad (X-a_{2g+2})(X-a_{2g})\vspace{-0.1mm}
\]
have coefficients in $\mathbb{F}_{p^2}$.
It follows that\vspace{-1.4mm}
\[
    a_{2g+2} = \frac{(a_{2g+2}+a_{2g})
    -(a_{2g}+a_{2g+1})
    +(a_{2g+1}+a_{2g+2})}{2}
    \in\mathbb F_{p^2}.
\]
After an $\mathbb{F}_{p^2}$-rational Möbius transformation sending $a_{2g+2}$ to $\infty$, we may therefore assume that $a_{2g+2} = \infty$.
Then $H$ is given by an odd-degree model\vspace{-1.3mm}
\[
    H:\ Y^2 = c\prod_{i=1}^{2g+1}(X-a_i),
    \quad c \in \mathbb F_{p^2}.\vspace{-4.5mm}
\]

\newpage
Recalling that $C$ and $H$ are isomorphic over $\overline{\mathbb{F}}_p$ (i.e., $C$ is a Rosenhain form of $H$), there exists a permutation $\sigma$ of $\{1,\ldots,2g+2\}$ such that\vspace{-1.5mm}
\[
    \lambda_i = \frac{a_{\sigma(i)}-a_{\sigma(2g)}}{a_{\sigma(i)}-a_{\sigma(2g+2)}} \cdot \frac{a_{\sigma(2g+1)}-a_{\sigma(2g+2)}}{a_{\sigma(2g+1)}-a_{\sigma(2g)}}\vspace{-0.4mm}
\]
for every $i \in \{1,\ldots,2g-1\}$.
Since $a_{2g+2} = \infty$, the right-hand side is a product of factors, each of which is either $1$ or a ratio of the form \eqref{eq:ratio}.
Therefore $\lambda_i$ is a square in $\mathbb{F}_{p^2}\hspace{-0.4mm}$ by Proposition~\ref{prop:4-torsion}.
In addition, a direct computation shows that,  for every $i,j \in \{1,\ldots,2g-1\}$,\vspace{-0.3mm}
\begin{align*}
    1-\lambda_i &= \frac{a_{\sigma(i)}-a_{\sigma(2g+1)}}{a_{\sigma(i)}-a_{\sigma(2g+2)}} \cdot \frac{a_{\sigma(2g)}-a_{\sigma(2g+2)}}{a_{\sigma(2g)}-a_{\sigma(2g+1)}},\\[0.2mm]
    \lambda_i-\lambda_j &= \frac{a_{\sigma(i)}-a_{\sigma(j)}}{a_{\sigma(i)}-a_{\sigma(2g+2)}} \cdot \frac{a_{\sigma(2g+2)}-a_{\sigma(2g)}}{a_{\sigma(2g+2)}-a_{\sigma(j)}} \cdot \frac{a_{\sigma(2g+1)}-a_{\sigma(2g+2)}}{a_{\sigma(2g+1)}-a_{\sigma(2g)}},
\end{align*}
whose right-hand sides are products of factors, each of which is either \(1\) or of the form \eqref{eq:ratio}.
Hence, applying Proposition~\ref{prop:4-torsion} once again, we conclude that $1-\lambda_i$ and $\lambda_i-\lambda_j$ are squares in $\mathbb{F}_{p^2}\hspace{-0.4mm}$ as well.
\end{proof}

\setcounter{equation}{0}
\section{Enumeration of superspecial hyperelliptic curves}\label{sec:algorithm}
In this section, we present an algorithm for enumerating superspecial genus-$g$  hyperelliptic curves in characteristic $p$ for fixed $(g,p)$, and report on the computational results obtained by implementing it.\vspace{-1mm}
\subsection{Algorithm}
The basic idea of our algorithm is simple.
We first enumerate all sets $\{\lambda_1,\ldots,\lambda_{2g-1}\}$ satisfying the condition of Theorem~\ref{thm:criterion}.
For each such set $\{\lambda_1,\ldots,\lambda_{2g-1}\}$, we then determine whether the corresponding hyperelliptic curve in Rosenhain form
\[
    C: y^2 = x(x-1)\prod_{i=1}^{2g-1}(x-\lambda_i) \eqqcolon f(x)\vspace{0.4mm}
\]
is superspecial; \hspace{0.3mm}equivalently, whether the Cartier-Manin matrix of $C$ is zero (cf. \cite[Theorem 4.1]{Nygaard}).
As is well known (cf. \cite[Section 2]{Yui}), the Cartier-Manin matrix of $C$ is given by\vspace{-0.6mm}
\[
    \begin{bmatrix}
        \gamma_{p-1} & \!\gamma_{p-2} & \!\cdots & \!\gamma_{p-g}\\
        \gamma_{2p-1} & \!\gamma_{2p-2} & \!\cdots & \!\gamma_{2p-g}\\
        \vdots & \!\vdots & \!\ddots & \!\vdots\\
        \gamma_{gp-1} & \!\gamma_{gp-2} & \!\cdots & \!\gamma_{gp-g}
    \end{bmatrix}\hspace{-0.6mm},
\]
where $\gamma_k$ denotes the coefficient of $x^k$ in the polynomial $f(x)^{(p-1)/2}$.
Therefore, it suffices to check whether the coefficient of $x^{ip-j}$ in $f(x)^{(p-1)/2}$ vanishes for every $i,j \in \{1,\ldots,g\}$.

We note that, in order to avoid enumerating isomorphic hyperelliptic curves multiple times, it is necessary to detect whether $C$ is isomorphic to a curve that has already been enumerated.
As explained at the beginning of Section \ref{sec:preliminaries}, this can be determined as follows.
Let\vspace{-2mm}
\[
    C': y^2 = x(x-1)\prod_{i=1}^{2g-1}(x-\lambda'_i)\vspace{-0.2mm}
\]
be another hyperelliptic curve in Rosenhain form, and define $a_1,\ldots,a_{2g+2}$ such that\vspace{-0.8mm}
\[
    \{a_1,\ldots,a_{2g+2}\} = \{\lambda_1,\ldots,\lambda_{2g-1}\} \sqcup \{0,1,\infty\}\vspace{0.1mm}
\]
Then $C$ and $C'$ are isomorphic to each other if and only if there exist pairwise distinct
$i,j,k \in\hspace{-0.2mm} \{1,\ldots,2g+2\}$ such that\vspace{0.6mm}
\begin{equation}\label{eq:isomorphic}
    \left\{
    \frac{a_\ell-a_i}{a_\ell-a_k} \cdot \frac{a_j-a_k}{a_j-a_i}\ \middle|\ \ell \in \{1,\ldots,2g+2\}\!\smallsetminus\!\{i,j,k\}\right\} = \{\lambda'_1,\ldots,\lambda'_{2g-1}\}.\vspace{0.5mm}
\end{equation}
Summarizing the above discussion, our algorithm can be described as Algorithm~\ref{alg:main} below.
\vspace{-1.5mm}

\newpage
\begin{Alg}\label{alg:main}
\ Input: An integer $g \geq 2$ and a prime $p > 2$.\\
Output: A list of representatives for isomorphism classes of superspecial genus-$g$ hyperelliptic curves over $\overline{\mathbb{F}}_p$.\vspace{-1mm}
\begin{enumerate}
\setlength{\leftskip}{22pt}
\item[{\it Step 1.}\,] Compute the set\vspace{-1.1mm}
\[
    S \coloneqq \{\lambda \in \mathbb{F}_{p^2}\hspace{-0.2mm} \mid \hspace{0.3mm}\text{both  $\lambda$ and $1-\lambda$ are squares in $\mathbb{F}_{p^2}$}\}\!\smallsetminus\!\{0,1\}.\vspace{-0.3mm}
\]
Initialize three empty lists $\hspace{0.2mm}\mathcal{C},\hspace{0.2mm}\mathcal{T}\hspace{-0.2mm}$, and $\mathcal{L}$.\vspace{-1.2mm}
\item[{\it Step 2.}\,] Recursively choose distinct $\lambda_1,\ldots,\lambda_{2g-1} \in S$ such that $\lambda_i-\lambda_j$ is a square in $\mathbb{F}_{p^2}\hspace{-0.3mm}$ for every $i < j$.
If the resulting set $\{\lambda_1,\ldots,\lambda_{2g-1}\}$ does not belong to $\mathcal{T}$, then add it to the list $\mathcal{C}$.
Moreover, for every choice of pairwise distinct $i,j,k \in \{1,\ldots,2g+2\}$, add the sets\vspace{-0.2mm}
\[
    \scalebox{0.9}{$\displaystyle\left\{
    \frac{a_\ell-a_i}{a_\ell-a_k} \cdot \frac{a_j-a_k}{a_j-a_i}\ \middle|\ \ell \in \{1,\ldots,2g+2\}\!\smallsetminus\!\{i,j,k\}\right\}$}\vspace{0.4mm}
\]
to the list $\mathcal{T}$, where $\{a_1,\ldots,a_{2g+2}\} = \{\lambda_1,\ldots,\lambda_{2g-1}\} \sqcup \{0,1,\infty\}$.\vspace{-1.2mm}
\item[{\it Step 3.}\,] For each element $\{\lambda_1,\ldots,\lambda_{2g-1}\}$ of $\mathcal{C}$, let\vspace{-1.4mm}
\[
    \scalebox{0.9}{$\displaystyle C: y^2 = x(x-1)\prod_{i=1}^{2g-1}(x-\lambda_i) \eqqcolon f(x).$}\vspace{-0.6mm}
\]
If the coefficient of $x^{ip-j}$ in $f(x)^{(p-1)/2}$ vanishes for every $i,j\in \{1,\ldots,g\}$, add $C$ to the list $\mathcal{L}$.\vspace{-1.2mm}
\item[{\it Step 4.}\,] Output $\mathcal{L}$.\vspace{-0.5mm}
\end{enumerate}
\end{Alg}

\subsection{Experimental results}
We implemented Algorithm~\ref{alg:main} in \textsf{Magma}, and the specific source code is available at the following URL:
\begin{center}
    \url{https://github.com/Ryo-Ohashi/EnumSSpHypCurves}.\vspace{0.6mm}
\end{center}
In this subsection, we present the experimental results obtained by executing our algorithm.
The experiments were conducted on a machine equipped with an AMD EPYC 7742 CPU and 2TB of RAM.\vspace{-1.5mm}

\subsubsection{The genus-4 case}
We executed Algorithm~\ref{alg:main} for all primes $p$ satisfying $11 \leq p \leq 41$ in the case $g=4$.
For $p \in\hspace{-0.2mm} \{11,13,17,19\}$, our results agree with those of Kudo-Harashita~\cite{KH18,KH22}.
For $p \geq 23$, we obtained the following results.\vspace{-0.6mm}
\begin{itemize}
    \item In characteristic $p=23$, the number of isomorphism classes of superspecial genus-4 hyperelliptic curves is equal to $4$. Specifically, these classes are represented by\vspace{-0.5mm}
    \begin{align*}
        y^2 &= x(x-1)(x-\zeta^2)(x-\zeta^{26})(x-\zeta^{34})(x-10)(x-\zeta^{190})(x-\zeta^{410})(x-\zeta^{454}),\\[-1mm]
        y^2 &= x(x-1)(x-\zeta^2)(x-\zeta^{26})(x-\zeta^{362})(x-\zeta^{364})(x-\zeta^{406})(x-15)(x-\zeta^{440}),\\[-1mm]
        y^2 &= x(x-1)(x-\zeta^2)(x-\zeta^{44})(x-\zeta^{78})(x-\zeta^{132})(x-\zeta^{318})(x-\zeta^{440})(x-\zeta^{482}),\\[-1mm]
        y^2 &= x(x-1)(x-\zeta^8)(x-\zeta^{34})(x-\zeta^{42})(x-\zeta^{130})(x-\zeta^{164})(x-\zeta^{190})(x-\zeta^{380}),
    \end{align*}
    where $\zeta \in \mathbb{F}_{23^2}\hspace{-0.2mm}$ is a root of the quadratic equation $z^2-2z+5 = 0$.\vspace{-1mm}
    \item In characteristic $p=29$, the number of isomorphism classes of superspecial genus-4 hyperelliptic curves is equal to $8$. Specifically, these classes are represented by\vspace{-0.5mm}
    \begin{align*}
        y^2 &= x(x-1)(x-\zeta^2)(x-\zeta^{10})(x-4)(x-8)(x-19)(x-\zeta^{408})(x-\zeta^{700}),\\[-1mm]
        y^2 &= x(x-1)(x-\zeta^2)(x-\zeta^{10})(x-\zeta^{84})(x-\zeta^{550})(x-\zeta^{610})(x-\zeta^{664})(x-\zeta^{760}),\\[-1mm]
        y^2 &= x(x-1)(x-\zeta^2)(x-\zeta^{14})(x-\zeta^{186})(x-\zeta^{212})(x-\zeta^{266})(x-\zeta^{292})(x-\zeta^{576}),\\[-1mm]
        y^2 &= x(x-1)(x-\zeta^2)(x-\zeta^{26})(x-\zeta^{28})(x-\zeta^{54})(x-\zeta^{286})(x-\zeta^{576})(x-\zeta^{788}),\\[-1mm]
        y^2 &= x(x-1)(x-\zeta^2)(x-\zeta^{26})(x-\zeta^{52})(x-\zeta^{54})(x-\zeta^{576})(x-\zeta^{662})(x-\zeta^{752}),\\[-1mm]
        y^2 &= x(x-1)(x-\zeta^2)(x-\zeta^{26})(x-\zeta^{52})(x-\zeta^{266})(x-\zeta^{482})(x-\zeta^{626})(x-\zeta^{790}),\\[-1mm]
        y^2 &= x(x-1)(x-\zeta^2)(x-\zeta^{28})(x-7)(x-\zeta^{576})(x-\zeta^{628})(x-\zeta^{656})(x-\zeta^{762}),\\[-1mm]
        y^2 &= x(x-1)(x-2)(x-4)(x-8)(x-6)(x-7)(x-26)(x-11)
    \end{align*}
    where $\zeta \in \mathbb{F}_{29^2}\hspace{-0.2mm}$ is a root of the quadratic equation $z^2-5z+2 = 0$.\vspace{-1mm}
    \newpage
    \item In characteristic $p=31$, the number of isomorphism classes of superspecial genus-4 hyperelliptic curves is equal to $10$. Specifically, these classes are represented by\vspace{-0.5mm}
    \begin{align*}
        y^2 &= x(x-1)(x-\zeta^6)(x-\zeta^{12})(x-\zeta^{18})(x-\zeta^{24})(x-16)(x-\zeta^{454})(x-\zeta^{486}),\\[-1mm]
        y^2 &= x(x-1)(x-\zeta^6)(x-\zeta^{18})(x-\zeta^{24})(x-\zeta^{220})(x-25)(x-\zeta^{436})(x-\zeta^{460}),\\[-1mm]
        y^2 &= x(x-1)(x-\zeta^6)(x-\zeta^{18})(x-\zeta^{92})(x-\zeta^{262})(x-\zeta^{520})(x-\zeta^{786})(x-\zeta^{792}),\\[-1mm]
        y^2 &= x(x-1)(x-\zeta^6)(x-\zeta^{18})(x-\zeta^{186})(x-\zeta^{504})(x-\zeta^{536})(x-\zeta^{678})(x-\zeta^{752}),\\[-1mm]
        y^2 &= x(x-1)(x-\zeta^6)(x-\zeta^{18})(x-\zeta^{220})(x-\zeta^{438})(x-\zeta^{446})(x-\zeta^{524})(x-\zeta^{764}),\\[-1mm]
        y^2 &= x(x-1)(x-\zeta^6)(x-\zeta^{18})(x-25)(x-\zeta^{446})(x-\zeta^{454})(x-\zeta^{740})(x-\zeta^{840}),\\[-1mm]
        y^2 &= x(x-1)(x-\zeta^8)(x-\zeta^{36})(x-\zeta^{110})(x-\zeta^{226})(x-\zeta^{302})(x-\zeta^{326})(x-\zeta^{696}),\\[-1mm]
        y^2 &= x(x-1)(x-\zeta^{24})(x-\zeta^{48})(x-\zeta^{216})(x-\zeta^{240})(x-\zeta^{478})(x-\zeta^{658})(x-2),\\[-1mm]
        y^2 &= x(x-1)(x-\zeta^{24})(x-\zeta^{100})(x-\zeta^{216})(x-\zeta^{240})(x-\zeta^{460})(x-\zeta^{764})(x-\zeta^{884}),\\[-1mm]
        y^2 &= x(x-1)(x-3)(x-9)(x-27)(x-20)(x-4)(x-15)(x-14),
    \end{align*}
    where $\zeta \in \mathbb{F}_{31^2}\hspace{-0.2mm}$ is a root of the quadratic equation $z^2-2z+3 = 0$.\vspace{-1mm}
    \item In characteristic $p=37$, the number of isomorphism classes of superspecial genus-4 hyperelliptic curves is equal to $23$. Specifically, these classes are represented by\vspace{-0.5mm}
    \begin{align*}
        y^2 &= x(x-1)(x-\zeta^2)(x-\zeta^{22})(x-2)(x-\zeta^{458})(x-\zeta^{692})(x-29)(x-\zeta^{1346}),\\[-1mm]
        y^2 &= x(x-1)(x-\zeta^2)(x-\zeta^{22})(x-\zeta^{178})(x-\zeta^{202})(x-\zeta^{258})(x-\zeta^{726})(x-\zeta^{728}),\\[-1mm]
        y^2 &= x(x-1)(x-\zeta^2)(x-\zeta^{24})(x-\zeta^{26})(x-\zeta^{202})(x-\zeta^{420})(x-\zeta^{974})(x-\zeta^{1192}),\\[-1mm]
        y^2 &= x(x-1)(x-\zeta^2)(x-\zeta^{24})(x-\zeta^{50})(x-\zeta^{180})(x-\zeta^{268})(x-\zeta^{574})(x-19),\\[-1mm]
        y^2 &= x(x-1)(x-\zeta^2)(x-\zeta^{24})(x-\zeta^{458})(x-15)(x-\zeta^{516})(x-20)(x-\zeta^{1112}),\\[-1mm]
        y^2 &= x(x-1)(x-\zeta^2)(x-\zeta^{26})(x-\zeta^{202})(x-\zeta^{282})(x-\zeta^{542})(x-\zeta^{592})(x-\zeta^{1134}),\\[-1mm]
        y^2 &= x(x-1)(x-\zeta^2)(x-\zeta^{36})(x-\zeta^{170})(x-34)(x-\zeta^{846})(x-\zeta^{946})(x-\zeta^{980}),\\[-1mm]
        y^2 &= x(x-1)(x-\zeta^2)(x-\zeta^{36})(x-\zeta^{170})(x-\zeta^{728})(x-\zeta^{946})(x-\zeta^{980})(x-\zeta^{1150}),\\[-1mm]
        y^2 &= x(x-1)(x-\zeta^2)(x-\zeta^{36})(x-\zeta^{268})(x-\zeta^{458})(x-\zeta^{838})(x-\zeta^{1028})(x-\zeta^{1294}),\\[-1mm]
        y^2 &= x(x-1)(x-\zeta^2)(x-2)(x-8)(x-\zeta^{396})(x-\zeta^{846})(x-\zeta^{854})(x-20),\\[-1mm]
        y^2 &= x(x-1)(x-\zeta^2)(x-2)(x-\zeta^{164})(x-\zeta^{204})(x-\zeta^{854})(x-\zeta^{980})(x-\zeta^{1110}),\\[-1mm]
        y^2 &= x(x-1)(x-\zeta^2)(x-2)(x-\zeta^{172})(x-\zeta^{426})(x-\zeta^{692})(x-\zeta^{730})(x-\zeta^{848}),\\[-1mm]
        y^2 &= x(x-1)(x-\zeta^2)(x-\zeta^{50})(x-\zeta^{268})(x-\zeta^{726})(x-\zeta^{1028})(x-\zeta^{1138})(x-\zeta^{1256}),\\[-1mm]
        y^2 &= x(x-1)(x-\zeta^2)(x-\zeta^{62})(x-\zeta^{166})(x-\zeta^{230})(x-\zeta^{900})(x-\zeta^{964})(x-\zeta^{980}),\\[-1mm]
        y^2 &= x(x-1)(x-\zeta^2)(x-\zeta^{62})(x-\zeta^{258})(x-\zeta^{458})(x-\zeta^{900})(x-20)(x-24),\\[-1mm]
        y^2 &= x(x-1)(x-\zeta^2)(x-4)(x-\zeta^{470})(x-15)(x-\zeta^{730})(x-\zeta^{964})(x-\zeta^{1200}),\\[-1mm]
        y^2 &= x(x-1)(x-\zeta^2)(x-\zeta^{112})(x-\zeta^{120})(x-\zeta^{524})(x-30)(x-\zeta^{642})(x-\zeta^{644}),\\[-1mm]
        y^2 &= x(x-1)(x-\zeta^8)(x-\zeta^{24})(x-\zeta^{88})(x-\zeta^{252})(x-\zeta^{260})(x-\zeta^{678})(x-35),\\[-1mm]
        y^2 &= x(x-1)(x-\zeta^8)(x-\zeta^{24})(x-35)(x-\zeta^{1116})(x-\zeta^{1280})(x-\zeta^{1288})(x-\zeta^{1318}),\\[-1mm]
        y^2 &= x(x-1)(x-\zeta^8)(x-\zeta^{34})(x-\zeta^{566})(x-\zeta^{784})(x-\zeta^{828})(x-\zeta^{1250})(x-\zeta^{1338}),\\[-1mm]
        y^2 &= x(x-1)(x-\zeta^{16})(x-\zeta^{36})(x-\zeta^{56})(x-\zeta^{236})(x-\zeta^{1124})(x-\zeta^{1148})(x-\zeta^{1324}),\\[-1mm]
        y^2 &= x(x-1)(x-\zeta^{16})(x-\zeta^{36})(x-\zeta^{60})(x-\zeta^{558})(x-\zeta^{678})(x-\zeta^{934})(x-\zeta^{978}),\\[-1mm]
        y^2 &= x(x-1)(x-2)(x-4)(x-8)(x-17)(x-23)(x-10)(x-7),
    \end{align*}
    where $\zeta \in \mathbb{F}_{37^2}\hspace{-0.2mm}$ is a root of the quadratic equation $z^2-4z+2 = 0$.\vspace{-1mm}
    \item In characteristic $p=41$, the number of isomorphism classes of superspecial genus-4 hyperelliptic curves is equal to $34$. Specifically, these classes are represented by\vspace{-0.5mm}
    \begin{align*}
        y^2 &= x(x-1)(x-\zeta^2)(x-\zeta^{18})(x-\zeta^{20})(x-\zeta^{220})(x-\zeta^{434})(x-34)(x-\zeta^{1480}),\\[-1mm]
        y^2 &= x(x-1)(x-\zeta^2)(x-\zeta^{18})(x-\zeta^{48})(x-28)(x-\zeta^{568})(x-\zeta^{946})(x-\zeta^{1412}),\\[-1mm]
        y^2 &= x(x-1)(x-\zeta^2)(x-\zeta^{18})(x-\zeta^{220})(x-\zeta^{434})(x-\zeta^{954})(x-20)(x-\zeta^{1484}),\\[-1mm]
        y^2 &= x(x-1)(x-\zeta^2)(x-\zeta^{18})(x-\zeta^{270})(x-\zeta^{708})(x-\zeta^{728})(x-\zeta^{946})(x-\zeta^{1464}),\\[-1mm]
        y^2 &= x(x-1)(x-\zeta^2)(x-\zeta^{20})(x-\zeta^{234})(x-\zeta^{392})(x-\zeta^{440})(x-18)(x-\zeta^{696}),\\[-1mm]
        y^2 &= x(x-1)(x-\zeta^2)(x-\zeta^{20})(x-\zeta^{392})(x-\zeta^{696})(x-\zeta^{1056})(x-\zeta^{1366})(x-\zeta^{1448}),\\[-6mm]
    \end{align*}
    \begin{align*}
        y^2 &= x(x-1)(x-\zeta^2)(x-\zeta^{22})(x-\zeta^{86})(x-\zeta^{390})(x-\zeta^{706})(x-\zeta^{846})(x-\zeta^{1656}),\\[-1mm]
        y^2 &= x(x-1)(x-\zeta^2)(x-\zeta^{22})(x-25)(x-\zeta^{316})(x-\zeta^{460})(x-28)(x-\zeta^{946}),\\[-1mm]
        y^2 &= x(x-1)(x-\zeta^2)(x-\zeta^{22})(x-\zeta^{616})(x-\zeta^{708})(x-\zeta^{886})(x-\zeta^{946})(x-\zeta^{976}),\\[-1mm]
        y^2 &= x(x-1)(x-\zeta^2)(x-\zeta^{24})(x-36)(x-\zeta^{268})(x-\zeta^{388})(x-\zeta^{602})(x-\zeta^{1296}),\\[-1mm]
        y^2 &= x(x-1)(x-\zeta^2)(x-\zeta^{24})(x-\zeta^{338})(x-\zeta^{602})(x-\zeta^{884})(x-\zeta^{1292})(x-\zeta^{1598}),\\[-1mm]
        y^2 &= x(x-1)(x-\zeta^2)(x-\zeta^{24})(x-\zeta^{388})(x-\zeta^{626})(x-\zeta^{710})(x-\zeta^{1088})(x-\zeta^{1466}),\\[-1mm]
        y^2 &= x(x-1)(x-\zeta^2)(x-\zeta^{32})(x-\zeta^{142})(x-\zeta^{386})(x-\zeta^{708})(x-\zeta^{1082})(x-\zeta^{1468}),\\[-1mm]
        y^2 &= x(x-1)(x-\zeta^2)(x-\zeta^{32})(x-\zeta^{244})(x-\zeta^{270})(x-\zeta^{728})(x-\zeta^{934})(x-\zeta^{1294}),\\[-1mm]
        y^2 &= x(x-1)(x-\zeta^2)(x-\zeta^{32})(x-\zeta^{386})(x-\zeta^{486})(x-\zeta^{740})(x-\zeta^{902})(x-\zeta^{1648}),\\[-1mm]
        y^2 &= x(x-1)(x-\zeta^2)(x-\zeta^{32})(x-\zeta^{708})(x-\zeta^{780})(x-\zeta^{856})(x-\zeta^{886})(x-\zeta^{1480}),\\[-1mm]
        y^2 &= x(x-1)(x-\zeta^2)(x-\zeta^{34})(x-25)(x-\zeta^{202})(x-\zeta^{478})(x-\zeta^{946})(x-\zeta^{1462}),\\[-1mm]
        y^2 &= x(x-1)(x-\zeta^2)(x-6)(x-\zeta^{120})(x-\zeta^{142})(x-25)(x-30)(x-\zeta^{1484}),\\[-1mm]
        y^2 &= x(x-1)(x-\zeta^2)(x-\zeta^{78})(x-\zeta^{416})(x-\zeta^{458})(x-18)(x-13)(x-\zeta^{1412}),\\[-1mm]
        y^2 &= x(x-1)(x-\zeta^2)(x-36)(x-\zeta^{282})(x-18)(x-34)(x-30)(x-2),\\[-1mm]
        y^2 &= x(x-1)(x-\zeta^2)(x-\zeta^{86})(x-\zeta^{212})(x-\zeta^{476})(x-\zeta^{884})(x-\zeta^{1220})(x-\zeta^{1484}),\\[-1mm]
        y^2 &= x(x-1)(x-\zeta^2)(x-\zeta^{134})(x-\zeta^{456})(x-\zeta^{520})(x-\zeta^{602})(x-\zeta^{736})(x-\zeta^{974}),\\[-1mm]
        y^2 &= x(x-1)(x-\zeta^2)(x-\zeta^{142})(x-\zeta^{282})(x-\zeta^{478})(x-\zeta^{748})(x-2)(x-\zeta^{1224}),\\[-1mm]
        y^2 &= x(x-1)(x-\zeta^2)(x-\zeta^{142})(x-\zeta^{476})(x-\zeta^{616})(x-\zeta^{1066})(x-2)(x-\zeta^{1346}),\\[-1mm]
        y^2 &= x(x-1)(x-\zeta^2)(x-\zeta^{146})(x-\zeta^{414})(x-\zeta^{434})(x-\zeta^{460})(x-\zeta^{846})(x-\zeta^{1414}),\\[-1mm]
        y^2 &= x(x-1)(x-\zeta^{16})(x-\zeta^{34})(x-\zeta^{286})(x-\zeta^{310})(x-\zeta^{730})(x-\zeta^{760})(x-\zeta^{1024}),\\[-1mm]
        y^2 &= x(x-1)(x-\zeta^{16})(x-\zeta^{34})(x-\zeta^{408})(x-\zeta^{472})(x-31)(x-\zeta^{1252})(x-\zeta^{1462}),\\[-1mm]
        y^2 &= x(x-1)(x-\zeta^{16})(x-6)(x-\zeta^{106})(x-\zeta^{148})(x-\zeta^{562})(x-\zeta^{1308})(x-\zeta^{1590}),\\[-1mm]
        y^2 &= x(x-1)(x-\zeta^{16})(x-6)(x-11)(x-\zeta^{432})(x-\zeta^{518})(x-40)(x-\zeta^{1442}),\\[-1mm]
        y^2 &= x(x-1)(x-\zeta^{16})(x-6)(x-11)(x-\zeta^{726})(x-\zeta^{746})(x-12)(x-\zeta^{1590}),\\[-1mm]
        y^2 &= x(x-1)(x-\zeta^{16})(x-\zeta^{92})(x-\zeta^{232})(x-\zeta^{520})(x-\zeta^{1268})(x-\zeta^{1288})(x-\zeta^{1484}),\\[-1mm]
        y^2 &= x(x-1)(x-\zeta^{16})(x-\zeta^{280})(x-\zeta^{476})(x-18)(x-\zeta^{936})(x-\zeta^{952})(x-\zeta^{1316}),\\[-1mm]
        y^2 &= x(x-1)(x-6)(x-36)(x-29)(x-28)(x-4)(x-26)(x-33),\\[-1mm]
        y^2 &= x(x-1)(x-6)(x-25)(x-39)(x-4)(x-3)(x-2)(x-20),
    \end{align*}
    where $\zeta \in \mathbb{F}_{41^2}\hspace{-0.2mm}$ is a root of the quadratic equation $z^2-3z+6 = 0$.\vspace{-0.5mm}
\end{itemize}
In particular, the above experimental results yield Theorem~\ref{thm:genus4}.
We note that these computations required a total of 28,196 seconds ($\approx$ 7.83 hours).\vspace{-1mm}

\subsubsection{The genus-5 case}
We executed Algorithm~\ref{alg:main} for all primes $p$ satisfying $13 \leq p \leq 41$ in the case $g=5$, obtaining the following results.\vspace{-0.7mm}
\begin{itemize}
\item In \hspace{-0.1mm}characteristic $\hspace{-0.2mm}p=13$, \hspace{-0.1mm}there is no superspecial genus-5 hyperelliptic curve.\vspace{-1.5mm}
\item In \hspace{-0.1mm}characteristic $\hspace{-0.2mm}p=17$, \hspace{-0.1mm}there is no superspecial genus-5 hyperelliptic curve.\vspace{-1.2mm}
\item In characteristic $p=19$, the number of isomorphism classes of superspecial genus-5 hyperelliptic curves is equal to $1$.
Specifically, this class is represented by\vspace{-0.6mm}
\[
    \scalebox{0.94}{$y^2 = x(x-1)(x-\zeta^4)(x-\zeta^{24})(x-\zeta^{44})(x-\zeta^{144})(x-\zeta^{244})(x-\zeta^{264})(x-\zeta^{284})(x-\zeta^{288})(x-\zeta^{324})$},\vspace{0.6mm}
\]
where $\zeta \in \mathbb{F}_{19^2}\hspace{-0.2mm}$ is a root of the quadratic equation $z^2-z+2 = 0$.\vspace{-1mm}
\item In characteristic $p=23$, the number of isomorphism classes of superspecial genus-5 hyperelliptic curves is equal to $2$.
Specifically, this class is represented by\vspace{-0.8mm}
\begin{align*}
    \scalebox{0.94}{$y^2\hspace{0.7mm}$} & \scalebox{0.94}{$=x(x-1)(x-\zeta^8)(x-\zeta^{34})(x-\zeta^{82})(x-\zeta^{130})(x-\zeta^{156})(x-\zeta^{164})(x-\zeta^{190})(x-\zeta^{346})(x-\zeta^{502})$},\\[-1mm]
    \scalebox{0.94}{$y^2\hspace{0.7mm}$} & \scalebox{0.94}{$=x(x-1)(x-5)(x-2)(x-10)(x-8)(x-11)(x-9)(x-22)(x-18)(x-15)$},
\end{align*}
where $\zeta \in \mathbb{F}_{23^2}\hspace{-0.2mm}$ is a root of the quadratic equation $z^2-2z+5 = 0$.
\newpage
\item In characteristic $p=29$, the number of isomorphism classes of superspecial genus-5 hyperelliptic curves is equal to $1$.
Specifically, this class is represented by\vspace{-0.7mm}
\[
    \scalebox{0.94}{$y^2 = x(x-1)(x-\zeta^8)(x-\zeta^{20})(x-\zeta^{84})(x-\zeta^{96})(x-\zeta^{260})(x-\zeta^{410})(x-\zeta^{530})(x-\zeta^{706})(x-\zeta^{726})$},\vspace{0.5mm}
\]
where $\zeta \in \mathbb{F}_{29^2}\hspace{-0.2mm}$ is a root of the quadratic equation $z^2-5z+2 = 0$.\vspace{-1mm}
\item In characteristic $p=31$, the number of isomorphism classes of superspecial genus-5 hyperelliptic curves is equal to $6$.
Specifically, this class is represented by\vspace{-0.7mm}
\begin{align*}
    \scalebox{0.94}{$y^2\hspace{0.7mm}$} & \scalebox{0.94}{$=x(x-1)(x-\zeta^6)(x-\zeta^{18})(x-\zeta^{24})(x-\zeta^{134})(x-\zeta^{244})(x-25)(x-\zeta^{436})(x-\zeta^{536})(x-\zeta^{722})$},\\[-1mm]
    \scalebox{0.94}{$y^2\hspace{0.7mm}$} & \scalebox{0.94}{$=x(x-1)(x-\zeta^6)(x-\zeta^{38})(x-\zeta^{66})(x-\zeta^{262})(x-\zeta^{326})(x-\zeta^{418})(x-\zeta^{436})(x-\zeta^{528})(x-\zeta^{700})$},\\[-1mm]
    \scalebox{0.94}{$y^2\hspace{0.7mm}$} & \scalebox{0.94}{$=x(x-1)(x-\zeta^6)(x-\zeta^{38})(x-\zeta^{66})(x-\zeta^{304})(x-\zeta^{446})(x-\zeta^{474})(x-\zeta^{524})(x-\zeta^{740})(x-\zeta^{826})$},\\[-1mm]
    \scalebox{0.94}{$y^2\hspace{0.7mm}$} & \scalebox{0.94}{$=x(x-1)(x-\zeta^6)(x-\zeta^{126})(x-\zeta^{202})(x-\zeta^{226})(x-\zeta^{418})(x-\zeta^{446})(x-\zeta^{520})(x-\zeta^{586})(x-\zeta^{662})$},\\[-1mm]
    \scalebox{0.94}{$y^2\hspace{0.7mm}$} & \scalebox{0.94}{$=x(x-1)(x-\zeta^8)(x-3)(x-\zeta^{134})(x-\zeta^{272})(x-10)(x-\zeta^{574})(x-\zeta^{666})(x-\zeta^{674})(x-6)$},\\[-1mm]
    \scalebox{0.94}{$y^2\hspace{0.7mm}$} & \scalebox{0.94}{$=x(x-1)(x-3)(x-9)(x-26)(x-8)(x-30)(x-28)(x-22)(x-5)(x-23)$},
\end{align*}
where $\zeta \in \mathbb{F}_{31^2}\hspace{-0.2mm}$ is a root of the quadratic equation $z^2-2z+3 = 0$.\vspace{-1mm}
\item In characteristic $p=37$, the number of isomorphism classes of superspecial genus-5 hyperelliptic curves is equal to $5$.
Specifically, this class is represented by\vspace{-0.7mm}
\begin{align*}
    \scalebox{0.94}{$y^2\hspace{0.7mm}$} & \scalebox{0.94}{$=x(x-1)(x-\zeta^2)(x-\zeta^{22})(x-2)(x-\zeta^{202})(x-\zeta^{258})(x-\zeta^{306})(x-\zeta^{728})(x-29)(x-\zeta^{1346})$},\\[-1mm]
    \scalebox{0.94}{$y^2\hspace{0.7mm}$} & \scalebox{0.94}{$=x(x-1)(x-\zeta^2)(x-\zeta^{24})(x-\zeta^{202})(x-\zeta^{458})(x-\zeta^{678})(x-\zeta^{934})(x-\zeta^{1112})(x-\zeta^{1134})(x-\zeta^{1136})$},\\[-1mm]
    \scalebox{0.94}{$y^2\hspace{0.7mm}$} & \scalebox{0.94}{$=x(x-1)(x-\zeta^2)(x-\zeta^{24})(x-\zeta^{258})(x-\zeta^{458})(x-\zeta^{678})(x-\zeta^{934})(x-\zeta^{1086})(x-\zeta^{1112})(x-\zeta^{1192})$},\\[-1mm]
    \scalebox{0.94}{$y^2\hspace{0.7mm}$} & \scalebox{0.94}{$=x(x-1)(x-\zeta^2)(x-\zeta^{24})(x-\zeta^{516})(x-\zeta^{662})(x-\zeta^{844})(x-\zeta^{1068})(x-\zeta^{1088})(x-\zeta^{1112})(x-\zeta^{1312})$},\\[-1mm]
    \scalebox{0.94}{$y^2\hspace{0.7mm}$} & \scalebox{0.94}{$=x(x-1)(x-\zeta^2)(x-\zeta^{26})(x-\zeta^{168})(x-\zeta^{202})(x-\zeta^{644})(x-\zeta^{810})(x-\zeta^{846})(x-\zeta^{1112})(x-\zeta^{1114})$},
\end{align*}
where $\zeta \in \mathbb{F}_{37^2}\hspace{-0.2mm}$ is a root of the quadratic equation $z^2-4z+2 = 0$.\vspace{-1mm}
\item In characteristic $p=41$, the number of isomorphism classes of superspecial genus-5 hyperelliptic curves is equal to $3$.
Specifically, this class is represented by\vspace{-0.7mm}
\begin{align*}
    \scalebox{0.94}{$y^2\hspace{0.7mm}$} & \scalebox{0.94}{$=x(x-1)(x-\zeta^2)(x-\zeta^{18})(x-\zeta^{270})(x-\zeta^{476})(x-\zeta^{568})(x-\zeta^{616})(x-\zeta^{954})(x-\zeta^{1388})(x-\zeta^{1412})$},\\[-1mm]
    \scalebox{0.94}{$y^2\hspace{0.7mm}$} & \scalebox{0.94}{$=x(x-1)(x-\zeta^2)(x-\zeta^{20})(x-6)(x-\zeta^{434})(x-18)(x-34)(x-9)(x-\zeta^{1336})(x-\zeta^{1480})$},\\[-1mm]
    \scalebox{0.94}{$y^2\hspace{0.7mm}$} & \scalebox{0.94}{$=x(x-1)(x-\zeta^{16})(x-6)(x-11)(x-\zeta^{432})(x-\zeta^{518})(x-\zeta^{810})(x-40)(x-\zeta^{886})(x-\zeta^{950})$},
\end{align*}
where $\zeta \in \mathbb{F}_{41^2}\hspace{-0.2mm}$ is a root of the quadratic equation $z^2-3z+6 = 0$.\vspace{-0.5mm}
\end{itemize}
In particular, the above experimental results yield Theorem~\ref{thm:genus5}.
We note that these computations required a total of 46,919 seconds ($\approx$ 13.03 hours).\vspace{-1mm}

\subsubsection{The genus-6 case}
We executed Algorithm~\ref{alg:main} for all primes $p$ satisfying $17 \leq p \leq 31$ in the case $g=6$, obtaining the following results.\vspace{-0.7mm}
\begin{itemize}
\item In \hspace{-0.1mm}characteristic $\hspace{-0.2mm}p=17$, \hspace{-0.1mm}there is no superspecial genus-6 hyperelliptic curve.\vspace{-1.5mm}
\item In \hspace{-0.1mm}characteristic $\hspace{-0.2mm}p=19$, \hspace{-0.1mm}there is no superspecial genus-6 hyperelliptic curve.\vspace{-1.2mm}
\item In characteristic $p=23$, the number of isomorphism classes of superspecial genus-6 hyperelliptic curves is equal to $1$.
Specifically, this class is represented by $y^2 = x(x-1)\prod_{i=1}^{11}(x-\lambda_i)$ with
\[
    \{\lambda_1,\ldots,\lambda_{11}\} = \{\zeta^2,\zeta^{26},\zeta^{74},\zeta^{122},\zeta^{146},\zeta^{338},\zeta^{362},\zeta^{410},\zeta^{458},\zeta^{482},\zeta^{484}\},\vspace{0.6mm}
\]
where $\zeta \in \mathbb{F}_{23^2}\hspace{-0.2mm}$ is a root of the quadratic equation $z^2-2z+5 = 0$.\vspace{-1mm}
\item In \hspace{-0.1mm}characteristic $\hspace{-0.2mm}p=29$, \hspace{-0.1mm}there is no superspecial genus-6 hyperelliptic curve.\vspace{-1.2mm}
\item In characteristic $p=31$, the number of isomorphism classes of superspecial genus-6 hyperelliptic curves is equal to $1$.
Specifically, this class is represented by $y^2 = x(x-1)\prod_{i=1}^{11}(x-\lambda_i)$ with
\[
    \{\lambda_1,\ldots,\lambda_{11}\} =  \{\zeta^6,\zeta^{12},\zeta^{24},\zeta^{214},\zeta^{226},\zeta^{386},\zeta^{486},\zeta^{586},\zeta^{746},\zeta^{758},\zeta^{948}\},\vspace{0.6mm}
\]
where $\zeta \in \mathbb{F}_{31^2}\hspace{-0.2mm}$ is a root of the quadratic equation $z^2-2z+3 = 0$.\vspace{-0.5mm}
\end{itemize}
In particular, the above experimental results yield Theorem~\ref{thm:genus6}.
We note that these computations required a total of 4,149 seconds ($\approx$ 1.15 hours).\vspace{2.3mm}

\textsc{Graduate School of Information Science and Technology, The
University of Tokyo — 7-3-1 Hongo, Bunkyo-ku, Tokyo, 113-0033, Japan.}\par
\textit{E-mail address}: \url{ryo-ohashi@g.ecc.u-tokyo.ac.jp}

\vspace{-4mm}

\newpage
\begin{thebibliography}{99}\vspace{-1mm}
\bibitem{Ekedahl} \textsc{T\hspace{-0.2mm}.\,Ekedahl}: \textit{On supersingular curves and abelian varieties}, Math.\hspace{0.8mm}Scand.\,{\bf 60}, No.\,2, 151--178, 1987.\vspace{-1mm}
\bibitem{Kudo} \textsc{M\hspace{-0.1mm}.\,Kudo}: \textit{Counting isomorphism classes of superspecial curves}, RIMS Kôkyûroku Bessatsu {\bf B90}, 77-95, 2022.\vspace{-1mm}
\bibitem{KH18} \textsc{M\hspace{-0.1mm}.\,Kudo, S.\,Harashita}: \textit{Superspecial hyperelliptic curves of genus 4 over small finite fields}, Arithmetic of finite fields, WAIFI 2018, LNCS \hspace{-0.2mm}{\bf 11321}, 58--73, 2018.\vspace{-1mm}
\bibitem{KH22} \textsc{M\hspace{-0.1mm}.\,Kudo, S.\,Harashita}: \textit{Algorithmic study of superspecial hyperelliptic curves over finite fields}, Comment.\,Math.\,Univ.\,St.\,Pauli {\bf 70}, 49--64, 2022.\vspace{-1mm}
\bibitem{NN} \textsc{M\hspace{-0.1mm}.\,S\hspace{-0.1mm}.\,Narasimhan, M\hspace{-0.1mm}.\hspace{0.1mm}V\hspace{-0.5mm}.\,Nori}: \textit{Polarisations on an abelian variety}, Proc.\,Indian Acad.\,Sci., Math.\,Sci. {\bf 90}, 125--128, 1981.\vspace{-1mm}
\bibitem{Nygaard} \textsc{N.\,O.\,Nygaard}: \textit{Slopes of powers of Frobenius on crystalline cohomology}, Ann.\,Sci.\,École Norm.\,Sup.\,{\bf 14}, No.\,4, 369--401, 1981.\vspace{-1mm}
\bibitem{Ohashi} \textsc{R.\,Ohashi}: \textit{On the Rosenhain forms of superspecial curves of genus two}, Finite Fields Appl.\,{\bf 97}, 102445, 2024.\vspace{-1mm}
\bibitem{OOKYN} \textsc{R.\,Ohashi, H\hspace{-0.1mm}.\,Onuki, M\hspace{-0.1mm}.\,Kudo, R.\hspace{0.4mm}Yoshizumi, K\hspace{-0.1mm}.\,Nuida}: \textit{Listing superspecial curves of genus three using Richelot isogeny graphs}, Res.\,Number \hspace{-0.2mm}Theory \hspace{-0.3mm}{\bf 11}, Paper No.\,76, 2025.\vspace{-1mm}
\bibitem{Valentini} \textsc{R.\,C.\,Valentini}: \textit{Hyperelliptic curves with zero Hasse-Witt matrix}, Manuscripta Math.\,{\bf 86}, No.\,2, 185--194, 1995.\vspace{-1mm}
\bibitem{Washington} \textsc{L.\,C.\hspace{0.3mm}Washington}, \textit{Elliptic Curves: Number Theory and Cryptography}, CRC, 2008.\vspace{-1mm}
\bibitem{Yui} \textsc{N\hspace{-0.1mm}.\hspace{0.3mm}Yui}: \textit{On the Jacobian varieties of hyperelliptic curves over fields of characteristic $p \hspace{-0.2mm}>\hspace{-0.2mm} 2$}, J.\,Algebra\,{\bf 52}, 378--410, 1978.\vspace{-1mm}
\bibitem{Zarhin} \textsc{Y\hspace{-0.2mm}.\,G.\,Zarhin}, \textit{Division by 2 on hyperelliptic curves and Jacobians}, arXiv:1606.05252.
\end{thebibliography}
\end{document}